\documentclass[11pt]{amsart}
\usepackage[colorlinks=true,pagebackref,hyperindex,linkcolor=Blue]{hyperref}
\usepackage[top=1in, bottom=1in, left=1.1in, right=1.1in]{geometry}
\usepackage{amsmath}
\usepackage{amsfonts}
\usepackage{amssymb}
\usepackage{amsthm}
\usepackage{stmaryrd}
\usepackage[all]{xy}  
\usepackage{mathrsfs}
\usepackage{enumitem}
\usepackage[T1]{fontenc}
\usepackage[normalem]{ulem}
\usepackage{color}

\makeatletter
\def\namedlabel#1#2{\begingroup
#2%
\def\@currentlabel{#2}%
\phantomsection\label{#1}\endgroup
}
\makeatother

\theoremstyle{theorem}
\newtheorem{theorem}{Theorem}[section]

\newtheorem{corollary}[theorem]{Corollary}
\newtheorem{lemma}[theorem]{Lemma}
\newtheorem{proposition}[theorem]{Proposition}

\newtheorem{theoremx}{Theorem}

\theoremstyle{definition}

\newtheorem{setup}[theorem]{Setup}

\newtheorem{example}[theorem]{Example}

\newtheorem{remark}[theorem]{Remark}
\numberwithin{equation}{subsection}

\newtheorem{questionx}[theoremx]{Question}

\newcommand{\set}[1]{\left\{ #1 \right\}}
\renewcommand{\(}{\left(}
\renewcommand{\)}{\right)}
\newcommand{\NN}{\mathbb{N}}
\newcommand{\ZZ}{\mathbb{Z}}
\newcommand{\QQ}{\mathbb{Q}}
\newcommand{\FF}{\mathbb{F}}

\newcommand{\cM}{\mathscr{M}}
\newcommand{\cN}{\mathscr{N}}

\newcommand{\cK}{K}
\newcommand{\ff}{\boldsymbol{f}}
\newcommand{\m}{\mathfrak{m}}
\newcommand{\n}{\mathfrak{n}}

\newcommand{\kk}{\Bbbk}
\newcommand{\ideala}{\mathfrak{a}}
\newcommand{\idealb}{\mathfrak{b}}

\newcommand{\idealc}{\mathfrak{c}}
\newcommand{\ideald}{\mathfrak{d}}
\newcommand{\f}{\boldsymbol{f} }

\renewcommand{\u}{\mathbf{u}}
\newcommand{\w}{\mathbf{w}}

\newcommand{\Spec}{\operatorname{Spec}}

\newcommand{\Ext}{\operatorname{Ext}}
\newcommand{\InjDim}{\operatorname{inj.dim}}
\newcommand{\Supp}{\operatorname{Supp}}

\newcommand{\Ker}{\operatorname{Ker}}
\newcommand{\Ass}{\operatorname{Ass}}
\newcommand{\IM}{\operatorname{Im}}

\newcommand{\FDer}[1]{\stackrel{#1}{\to}}

\newcommand{\connarrow}{\ar `[d] `[l] `[llld] `[ld] [lld]}

\newcommand{\newLyu}{\widetilde{\lambda}}

\definecolor{blue-violet}{rgb}{0.54, 0.17, 0.89}
\definecolor{Blue}{rgb}{0.01, 0.28, 1.0}
\definecolor{gGreen}{rgb}{0.2, 0.8, 0.2}
\definecolor{Green}{rgb}{0.04, 0.85, 0.32}

\newcommand{\seg}[1]{#1}

\newcommand{\name}[1]{{\slshape #1}}

\begin{document}

\title[Lyubeznik numbers and injective dimension]{Lyubeznik numbers and injective dimension in\\ mixed characteristic}
\author[D.\ J.\ Hern\'andez]{Daniel J.\ Hern\'andez}
\author[L.\ N\'u\~nez-Betancourt]{Luis N\'u\~nez-Betancourt}
\author[F.\ P\'erez]{Felipe P\'erez}
\author[E.\ E.\ Witt]{Emily E.\ Witt}
\maketitle


\begin{abstract}
We investigate the Lyubeznik numbers, and the injective dimension of local cohomology modules, of finitely generated $\mathbb{Z}$-algebras. 
We prove that the mixed characteristic Lyubeznik numbers and the standard ones agree locally for almost all reductions to positive characteristic. 
Additionally, we address an open question of Lyubeznik that asks whether the injective dimension of a local cohomology module over a regular ring is bounded above by the dimension of its support. 
Although we show that the answer is affirmative for several families of $\mathbb{Z}$-algebras, we also exhibit an example where this bound fails to hold.  
This example settles Lyubeznik's question, and illustrates one way that the behavior of local cohomology modules of regular rings of equal characteristic and of mixed characteristic can differ.
\end{abstract}


\section{Introduction}


Since its introduction by Grothendieck in the 1960s, the theory of local cohomology has been an active area of research in commutative algebra,  algebraic geometry, and neighboring fields.  In this paper, we are especially motivated by the following observation:  despite the fact that local cohomology modules of regular rings are typically large (e.g., not finitely generated), they often satisfy strong structural conditions.  More precisely, consider the following conditions on a module $M$ over a ring $R$.

\begin{enumerate}
\item \label{F1} The number of associated primes of $M$ is finite.
\item \label{F2} The Bass numbers of $M$ are finite.
\item \label{F3} $\InjDim_R M \leq \dim \Supp_R M$.\footnote{To avoid contemplating the dimension of $\varnothing$, we  only consider property \eqref{F3} for nonzero $M$.}
\item \label{F4} The local cohomology of $M$ with support in a maximal ideal of $R$ is injective.
\end{enumerate}

\noindent In condition  \eqref{F3} above, we use $\InjDim_R M$ to denote the injective dimension of $M$ (i.e., the length of a minimal injective $R$-resolution of $M$), and $\dim \Supp_R M$ to denote the dimension of the support of $M$ as a topological subspace of $\Spec R$ (i.e., the longest length of a chain of prime ideals in $R$ whose terms are contained in the support of $M$.)

It was shown by Huneke and Sharp in positive characteristic \cite{Huneke}, and by Lyubeznik in equal characteristic zero \cite{LyuDMod} and in unramified mixed characteristic \cite{LyuUMC}, that the local cohomology modules of a regular local ring with support in an arbitrary ideal satisfy conditions \eqref{F1} and \eqref{F2} above.  More recently, it was shown by Bhatt, Blickle, Lyubeznik, Singh, and Zhang that the local cohomology modules of a smooth $\ZZ$-algebra with support in an arbitrary ideal also satisfy condition \eqref{F1} \cite{Todos}.  The extent to which local cohomology modules satisfy \eqref{F3} and \eqref{F4} is a subtle issue, and an overarching theme of this paper;  we recall results regarding conditions \eqref{F3} and \eqref{F4} by Huneke and Sharp, Lyubeznik, and Zhao later in this introduction.

The finiteness condition \eqref{F2} for local cohomology is a crucial ingredient in the definition of the so-called \emph{Lyubeznik numbers}, a family of integer-valued invariants associated to a local ring containing a field.  In addition to encoding important properties of the ambient ring \cite{Kawasaki2,Walther2,Kawasaki,LyuInvariants,W},  Lyubeznik numbers also have interesting geometric and topological interpretations, including connections with singular and \'etale cohomology \cite{GarciaSabbah,B-B} and with simplicial complexes \cite{AGZ, LyuNumMontaner, AMY}.  
It is worth pointing out that properties \eqref{F1}, \eqref{F3}, and \eqref{F4} play an important role in the proofs of these results.  
The finiteness property \eqref{F2} is also used to define the \emph{mixed characteristic Lyubeznik numbers}\footnote{In \cite{NuWiMixChar,SurveyLyuNum}, these invariants are referred to as \emph{Lyubeznik numbers in mixed characteristic}.}, a variant of the (standard) Lyubeznik numbers defined by the second and fourth authors \cite{NuWiMixChar}.
A major obstruction to understanding these invariants has been the uncertainty surrounding properties \eqref{F3} and \eqref{F4} in mixed characteristic.

This article has two major goals:  the first is to investigate the relationship between the (standard) Lyubeznik numbers and the mixed characteristic Lyubeznik numbers in the case in which they are both defined (that is, for local rings of positive characteristic), and the second is to investigate when local cohomology modules satisfy conditions \eqref{F3} and \eqref{F4}.


\subsection{A comparison of Lyubeznik numbers}

Suppose that $(R, \m, \kk)$ is a local ring.  If $R$ contains a field, and may be realized as $R \cong S/I$, where $S$ is a regular local ring containing a field, then the (standard) \emph{Lyubeznik number} of $R$ with respect to integers $i,j \geq 0$ is defined as the $i^{\text{th}}$ Bass number of the local cohomology module $H^{\dim(S) - j}_I(S)$; i.e.,  \[ \lambda_{i,j}(R) = \dim_{\kk} \Ext_S^i(\kk, H_I^{\dim(S) - j}(S)).\]  

In an attempt to study local rings of mixed characteristic (that is, rings of characteristic zero with positive characteristic residue fields), the second and fourth authors proposed the following analogous definition: if the residue field $\kk$ of $R$ has characteristic $p>0$, and $R$ may be realized as $R \cong T/J$, where $T$ is an {unramified} regular local ring of mixed characteristic $p>0$, then the \emph{mixed characteristic Lyubeznik number}\footnote{This terminology is \emph{not} intended to imply that $R$ must be of mixed characteristic in order to define $\newLyu_{i,j}(R)$ (this would be false), but is chosen to emphasize that the integers $\newLyu_{i,j}(R)$ arise from considering $R$ (or $\widehat{R}$) as a quotient of a regular unramified ring of mixed characteristic.} of $R$ with respect to integers $i, j \geq 0$ is defined as 
\[ \newLyu_{i,j}(R)= \dim_{\kk} \Ext_T^i(\kk,H_J^{\dim(T)-j}(T)). \] 

One may then extend these definitions to an arbitrary local ring $(R, \m, \kk)$ by defining \[ \lambda_{i,j} (R) = \lambda_{i,j}(\widehat{R}) \text{ and } \newLyu_{i,j}(R) = \newLyu_{i,j}(\widehat{R}).\] 
(The fact that $\widehat{R}$ can be realized in the specified ways follows from the Cohen structure theorems \cite{Cohen}).  In either case, the Lyubeznik numbers depend only on $R$ and the indices $i$ and $j$, but not on any of the choices made realizing $R$ (or $\widehat{R}$) as a quotient \cite{LyuDMod,NuWiMixChar}.

It is important to note that, although the mixed characteristic Lyubeznik numbers were originally defined to study rings of mixed characteristic, both the standard and the mixed characteristic Lyubeznik numbers are  defined for rings of (equal) characteristic $p>0$.  This observation motivates the following question.

\begin{questionx}
\label{Q LyuNum Intro}
For which rings of positive characteristic is it true that the (standard) Lyubeznik numbers and the mixed characteristic Lyubeznik numbers agree?
\end{questionx}

It is known that Question \ref{Q LyuNum Intro} has a positive answer for Cohen-Macaulay rings and rings of small dimension \cite[Corollary 5.3]{NuWiMixChar}.  
Interestingly enough, however, these types of Lyubeznik numbers need \emph{not} always agree, as exhibited by a certain quotient of a power series over a field of characteristic two by a squarefree monomial ideal \cite[Remark 6.11]{NuWiMixChar}.  

In this article, we consider Question \ref{Q LyuNum Intro} from the perspective of reduction to positive characteristic.  Recall that if $R$ is a finitely generated $\ZZ$-algebra, then we call \[ \FF_p \otimes_{\ZZ} R = R / p R \] the reduction of $R$ to characteristic $p>0$.  In this context, one may specialize Question \ref{Q LyuNum Intro}, and instead ask for which reductions of a given finitely generated $\ZZ$-algebra do the standard and mixed characteristic Lyubeznik numbers agree locally.  
Before stating our general result in this direction, we present an illustrative example.

\begin{example}[Cohen-Macaulay rings] \label{CM: e}  If a finitely-generated $\ZZ$-algebra $R$ is Cohen-Macaulay, then so are its reductions $R/pR$ for all $p \gg 0$ \cite[Theorem 2.3.5]{HHCharZero}.  In light of this, the aforementioned result {\cite[Corollary 5.3]{NuWiMixChar}} shows that if $p \gg 0$ and $A$ is the localization of $R/pR$ at a prime ideal, then 
$\newLyu_{i,j} \(A \) = \lambda_{i,j}\( A \)$ for all integers $i, j \geq 0$.  In fact, the vanishing of local cohomology modules described in \cite[Theorem III.4.1]{P-S} implies that both of these invariants are zero unless $j=\dim(R)$.
\end{example}

As suggested by this example, the standard and mixed characteristic Lyubeznik numbers agree locally for almost all reductions to characteristic $p>0$ of a given finitely-generated $\ZZ$-algebra.

\begin{theoremx}[cf.\,Theorem \ref{Thm Lyubeznik Numbers Refined}]\label{Thm Lyubeznik Numbers}  If $R$ is a finitely generated $\ZZ$-algebra, then there exists a finite set of prime integers $W$ with the following property:  If $p$ is prime and not contained in $W$, and $A$ is the localization of $\FF_p \otimes_{\ZZ} R$ at a prime ideal, then $\newLyu_{i,j}(A) = \lambda_{i,j}(A)$ for all integers $i,j \geq 0$.
\end{theoremx}

\noindent A fundamental tool in our proof of Theorem \ref{Thm Lyubeznik Numbers} is the fact
that local cohomology modules of polynomial rings over $\ZZ$ have finitely many associated primes \cite{Todos}.

\subsection{The injective dimension of local cohomology modules}  

In light of the fact that the local cohomology of a regular ring containing a field satisfies the structural conditions \eqref{F3} and \eqref{F4},
the following question of Lyubeznik \cite{LyuDMod,LyuUMC} is natural.

\begin{questionx}[Lyubeznik] \label{Q InjDim}  Given an ideal $I$ of a regular local ring $S$, do the conditions \eqref{F3} and \eqref{F4} hold for nonzero $H^j_I(S)$?  
\end{questionx}

This question has been unresolved in the mixed characteristic case for more than two decades.  To the best of our knowledge, the most significant step toward answering Question \ref{Q InjDim} is the result of Zhou \cite[Theorem 5.1]{Zhou}, which says the following:  If $S$ is regular of unramified mixed characteristic, and $H^j_I(S)$ is nonzero for some ideal $I$ of $S$ and integer $j \geq 0$, then the injective dimension of $H^j_I(S)$ is bounded above by $\dim\Supp_S H^j_I(S)+1$.  Furthermore, if $\m$ is a maximal ideal of $S$, then the injective dimension of the iterated local cohomology module $H^i_{\m} H^j_I(S)$ is at most one.  In this article, we prove that the answer to the first part Question C is affirmative, in most cases, for $\ZZ$-algebras that descend from $\QQ$-algebras.

\begin{theoremx}[cf.\,Theorem \ref{positiveInjDim: T}] \label{Thm Injective Module}  
Given an ideal $I$ of $S = \ZZ[x_1, \ldots, x_n]$ with $\QQ \otimes_{\ZZ} H^j_I(S)\neq 0$,  there exists a finite set of primes $W$ with the following property: If $Q$ is a prime ideal in $\Supp_S H^j_I(S)$ not lying over any prime in $W$\footnote{The condition that $\QQ \otimes_{\ZZ} H^j_I(S) \neq 0$ and the finiteness of $W$ imply that the set of all such prime ideals $Q$ is a non-empty open (i.e., dense) subset of $\Supp_S H^j_I(S)$.  See Subsection \ref{DescendsFromQQ: ss} for more details.}, then $\InjDim_S H^j_{I}(S)_Q \leq \dim \Supp_S H^j_{I}(S)_Q$.  
\end{theoremx}

We also show that local cohomology modules satisfy condition \eqref{F4} in several cases; see Proposition \ref{ILC: p} for details.  
Moreover, although we have identified many cases with a positive answer to Question \ref{Q InjDim}, we also construct an example in which properties  \eqref{F3} and \eqref{F4} fail to hold (cf. the example presented in \cite[Remark 6.11]{NuWiMixChar}).    

\begin{theoremx}[cf.\,Theorem \ref{exactCounterexample: T}] \label{Thm Counter-Example}
There exists an ideal $I$ of a regular local ring $(T, \m)$ of mixed characteristic, and nonnegative integers $i$ and $j$,
for which $\InjDim_S H^j_I(T)>\dim\Supp_S H^j_I(T)$, and such that the iterated local cohomology module $H^i_{\m} H^j_I(T)$ is not injective.
\end{theoremx} 
\noindent This theorem establishes a negative answer to both parts of Question \ref{Q InjDim}, effectively settling it.
Furthermore, it exhibits a previously-unknown way in which regular rings of equal characteristic and those of mixed characteristic can behave differently.

We conclude this introduction by mentioning one further result, which may be regarded as an extension of the recent noted work on smooth $\ZZ$-algebras \cite{Todos}; see Section \ref{Bock} for the proof.

\begin{theoremx} \label{Thm Ass}
Given ideals $I_1,\ldots,I_t$ of $S=\ZZ[x_1,\ldots,x_n]$, and nonnegative integers $j_1,\ldots,j_t$, the $t$-fold iterated local cohomology module $H^{j_t}_{I_t}\cdots H^{j_1}_{I_1}(S)$ has only finitely many associated primes.
\end{theoremx}

We note that Theorem \ref{Thm Ass}, and the ideas behind the proof (namely, Corollary \ref{finitelyManyBadPrimes: C}) play a key role in our proofs of Theorem \ref{Thm Lyubeznik Numbers}, Theorem \ref{Thm Injective Module}, and Proposition \ref{ILC: p}.

\section{Local cohomology} 

In this section, we recall some basic properties of local cohomology modules, and refer the reader to \cite{BroSharp} for more details.


\subsection{Koszul and local cohomology} 
\label{Loc Coh} 
The (cohomological) \name{Koszul complex} of an element $f$ of a commutative ring $R$ is the complex $\cK^{\bullet}(f; R)$ given by $0 \longrightarrow R \ {\longrightarrow} \ R \longrightarrow 0$, where the non-trivial map is multiplication by $f$, the left-most copy of $R$ is in degree zero, and the right-most copy in degree one.  More generally, the Koszul complex of a finite sequence $\f = \langle f_1, \cdots, f_s \rangle$ in $R$ on an $R$-module $M$ is defined to be the tensor product of complexes \[  \cK^{\bullet}( \f; M) : = \cK^{\bullet}(f_1;R)\otimes_R \cdots \otimes_R \cK^{\bullet}(f_s;R) \otimes_R M, \] where we regard $M$ as a complex concentrated in degree zero.  The \name{$k$-th Koszul cohomology module of $\f$} is then defined to be $H^k(\f; R) := H^k(K^\bullet (\boldsymbol{f}; R))$.

Given a non-negative integer  $t$, the commutative diagram of $R$-modules 
 \begin{equation}
 \label{mapOnKoszul: e}
 \xymatrix@R=.8cm@C=1cm
{0 \ar[r] & R    \ar@{=}[d] \ar[r]^{f^t}  & R \ar[r] \ar[d]^-{f} & 0  \\  
0 \ar[r] & R   \ar[r]^{f^{t+1}} & R \ar[r] & 0} 
\end{equation}
defines a map of complexes $K^{\bullet}(f^t; R) \to K^{\bullet}(f^{t+1}; R)$, and tensoring maps of this form together, and then tensoring with $M$, produces a map 
\begin{equation}
\label{DirSysDefiningLC: e}
K^\bullet(\ff^t; M) \to K^\bullet(\ff^{t+1}; M),
\end{equation}
where $\f^n$ is the sequence obtained by raising every term in $\f$ to the $n$-th power.   We use $K^{\bullet}(\ff^{\infty}; M)$ to denote the limit complex of this directed system, and we recall that there exists a canonical isomorphism $K^{\bullet}(\ff^{\infty}; M)  \cong \Check{C}^{\bullet}(\f; M)$, 
where $\Check{C}^{\bullet}(M; \f)$ denotes the \name{\v{C}ech complex} of $\ff$ on $M$.  Taking homology in \eqref{DirSysDefiningLC: e} produces a directed system of $R$-modules, and as direct limits commute with homology, we obtain a functorial isomorphism 
\begin{equation}
\label{definitionLC: e}
\lim \limits_{t \to \infty} H^k( K^{\bullet}(\f^t; M) ) \cong H^k( K^{\bullet}(\f^{\infty}; M)  )  \cong H^k( \Check{C}^{\bullet}(\f; M) ).
\end{equation}

These $R$-modules depend only on the ideal $\ideala$ generated by the terms in $\f$, and we refer to these isomorphic objects, denoted by $H^k_{\ideala}(M)$, as the \name{$k$-th local cohomology module of $M$ with support in $\ideala$}.  

\subsection{Long exact sequences in local cohomology}

Assume now that $R$ is Noetherian.
Given an ideal $\ideala$ of $R$ and a short exact sequence $0 \to M \to N \to P \to 0$ of $R$-modules, there is a functorial long exact sequence
\[ \cdots \to H^k_\ideala(M) \to  H^k_\ideala(N) \to  H^k_\ideala(P) \to H^{k+1}_\ideala(M) \to \cdots. \]
Given $f \in R$ and an $R$-module $M$, there is also a long exact sequence, functorial in $M$, 
\begin{equation*} 
\cdots \to H^k_{\ideala+ f R}(M) \to H^k_\ideala(M) \to H^k_\ideala(M_f) \to H^{k+1}_{\ideala+ f R}(M) \to\cdots,
\end{equation*}
where the map from $H^k_\ideala(M)$ to $H^k_\ideala(M_f) \cong H^k_\ideala(M)_f$ is the natural localization map.

\section{Finiteness of associated primes of iterated local cohomology modules}\label{Bock}


The goal of this section is to prove Theorem \ref{Thm Ass}, which extends the recent result of Bhatt, Blickle, Lyubeznik, Singh, and Zhang on the finiteness of the associated primes of local cohomology \cite[Theorem 1.2]{Todos}.  The proof itself appears in the last subsection, while in the earlier subsections, we survey some results from the theory of $F$-modules and $D$-modules.  Our main reference for the theory of $F$-modules is \cite{LyuFmod}.


\subsection{$F$-modules}
\label{F-mod} 
Given a ring $R$ of characteristic $p>0$ and a non-negative integer $e$, let $R^{(e)}$ denote the abelian group $R$ considered as an $R$-bimodule via the rule  $r \cdot x \cdot s  = rxs^{p^e}$.  For the remainder of this subsection, we assume that $R$ is regular, so that the functor $F$ taking an $R$-module $M$ to the (left) $R$-module $F(M) = R^{(e)} \otimes_R M$ is exact.  An \name{$F$-module} consists of an $R$-module $M$ and an $R$-module isomorphism $M \cong F(M)$, which we call the \name{structure morphism} of $M$.  The isomorphism $R \to F(R)$ given by $x \mapsto x \otimes 1$ shows that $R$ is an $F$-module, and in this article, we always regard $R$ as an $F$-module in this way.  A map between $F$-modules $M \cong F(M)$ and $N \cong F(N)$ is an $R$-linear map $M \to N$ that respects the structure morphisms (that is, so that the expected diagram commutes).

Given a map $\alpha: M \to F(M)$ of $R$-modules,  consider the following commutative diagram.
\[ \xymatrix@C=1.5cm@R=1cm{ M \ar[r]^{\alpha} \ar[d]_{\alpha} & F(M) \ar[r]^{F(\alpha)} \ar[d]^{F(\alpha)} & F^2(M) \ar[d]^{F^2(\alpha)} \ar[r]^{\ \ \ F^2(\alpha)}  & \cdots  \\ 
F(M) \ar[r]^{F(\alpha)}  & F^2(M) \ar[r]^{F^2(\alpha)} & F^3(M) \ar[r]^{ \ \ \ F^3(\alpha)}  & \cdots }\]

If $\cM$ denotes the direct limit of the first row of this diagram, the vertical arrows induce an isomorphism $\cM \rightarrow F(\cM)$.  In this case, we say that $\alpha$ \name{generates} the $F$-module $\cM$.

\subsection{Local cohomology as an $F$-module}

The localization of an $F$-module can be regarded as an $F$-module in such a way that the localization map is a map of $F$-modules \cite[Example 1.2]{LyuFmod}.  In particular, the \v{C}ech complex of an $F$-module is a complex of $F$-modules, and therefore, the (iterated) local cohomology of an $F$-module is also an $F$-module.  In Proposition \ref{generating: P} below, which is an adaptation of \cite[Proposition 1.11(b)]{LyuFmod}, we describe another (isomorphic) $F$-module structure on local cohomology.  Before proceeding, we recall some basic facts:  

\begin{remark}[Direct limits and tensor products] 
\label{directLimits: R}
The natural numbers are a basic example of a \emph{filtered poset}.  As such, direct limits of $\NN$-directed systems satisfy certain desirable conditions.  For example, if $\set{M_i}_{i=1}^{\infty}$ and $\set{N_i}_{i=1}^{\infty}$ are two directed systems of $R$-modules, then 
\[ \( \lim_{i \to \infty} N_i \) \otimes_R \( \lim_{i  \to \infty} M_i \) \cong \lim_{i  \to \infty} ( M_i \otimes_R N_i ).\] 
The analogous identity for complexes of $R$-modules also holds \cite[Theorem 4.28]{TwentyFourHours}, and we use these identities without mention in the following remark.
\end{remark}

\begin{remark}[Koszul complexes and the Frobenius functor]
\label{koszul+frobenius: R}
The standard $F$-module structure $R \cong F(R)$ allows us to canonically identify $K^{\bullet}(f^p; R)$ with $F(K^{\bullet}(f; R))$, which gives us two different ways to regard the second row in the following commutative diagram  %
 \[ 
 \xymatrix@R=.8cm@C=1cm
{0 \ar[r] & R    \ar@{=}[d] \ar[r]^{f}  & R \ar[r] \ar[d]^-{f^{p-1}} & 0  \\  
0 \ar[r] & R   \ar[r]^{f^{p}} & R \ar[r] & 0}. 
\]
In other words, there exists a commutative diagram of complexes 
%
\[  \xymatrix@R=.8cm@C=1cm{ 
& K^{\bullet}(f; R)  \ar[dr] \ar[dl] &   \\  K^{\bullet}(f^p; R)  \ar[rr]^-{\cong} & & F ( K^{\bullet}(f; R)).
} \]

More generally, given a finite sequence $\f$ in $R$, the flatness of the Frobenius morphism allows us tensor sequences of this form together to obtain a commutative diagram
\[ \xymatrix@R=.8cm@C=1cm{ 
& K^{\bullet}(\ff; R)  \ar[dr] \ar[dl] &   \\  K^{\bullet}(\ff^p; R)  \ar[rr]^-{\cong} & & F ( K^{\bullet}(\ff; R)).
} \]

Next, fix an $R$-module $M$, and let $\cM$ be the $F$-module generated by $\alpha: M \to F(M)$.    Tensoring the vertical maps above with $\alpha$ produces
\begin{equation}
\label{koszulTriangle: e}
   \xymatrix@R=.8cm@C=1cm{ 
& K^{\bullet}(\ff; M)  \ar[dr] \ar[dl] &   \\  K^{\bullet}(\ff^p; F(M))  \ar[rr]^-{\cong} & & F ( K^{\bullet}(\ff; M)), 
} 
\end{equation}
and considering the directed systems that result from iterating these maps, we see that
\begin{align*}
\lim_{e \to \infty}  F^e(K^{\bullet}(\ff; M)) \cong \lim_{e \to \infty} K^{\bullet}(\ff^{p^e}, F^e(M)) & = \lim_{e \to \infty} \( K^{\bullet}(f^{p^e}; R) \otimes F^e(M) \) \\ 
& \cong \( \lim_{e \to \infty}  K^{\bullet}(\ff^{p^e}; R) \) \otimes \( \lim_{e \to \infty} F^e(M) \) \\
& = \( \lim_{e \to \infty} K^{\bullet}(\ff^{p^e}; R) \) \otimes \cM \\
& \cong \lim_{e \to \infty}  \(  K^{\bullet}(\ff^{p^e}; R) \otimes \cM \) = K^{\bullet}(\ff^{\infty}; \cM).  
\end{align*}
At the level of homology, the isomorphism $K^{\bullet}(\ff^{\infty}, \cM) \cong \lim_{e \to \infty} F^e( K^{\bullet}(\ff; M))$ becomes 
\[ H^k_{\ideala}(\cM) = H^k (K^{\bullet}(\ff^{\infty}, \cM)) \cong  \lim_{e \to \infty} H^k( F^e( K^{\bullet}(\ff; M)) \cong \lim_{e \to \infty} F^e( H^k(\ff; M)),  \] where $\ideala$ is the ideal generated by the terms of $\ff$, and the last directed system is obtained by taking homology of the right-most map in \eqref{koszulTriangle: e}.   
\end{remark}

We summarize the content of Remark \ref{koszul+frobenius: R} below.

\begin{proposition}[cf.\,{\cite[Proposition 1.11(b)]{LyuFmod}}]
\label{generating: P}
Fix a finite sequence $\f$ in $R$ whose terms generate the ideal $\ideala$ of $R$. Fix an arbitrary $R$-linear homomorphism $\alpha : M \to F(M)$, and let $\beta: H^k(\f; M) \to F(H^k(\f; M))$ be the associated map induced by \eqref{koszulTriangle: e}.  If $\cM$ is the $F$-module generated by $\alpha$, and $\cN$ is the $F$-module generated by $\beta$, then there exists an isomorphism of $F$-modules $\cN \cong H^k_{\ideala}(\cM)$ such that \[ \xymatrix@R=.8cm@C=1cm
{ H^k(\f; M)  \ar[r]  \ar[d] & H^k(\f ; \cM) \ar[d] \\ \cN \ar[r]^-{\cong} & H^k_{\ideala}(\cM) }\]
commutes, where $H^k(\f; M) \to \cN$ is the map to the direct limit, $H^k(\f; M) \to H^k(\f; \cM)$ is the functorial map induced by the map to the direct limit $M \to \cM$, and the right-most vertical map is the one given by the natural transformation from Koszul to local cohomology.
\end{proposition} 

\subsection{Iterated local cohomology}  For the rest of this section, we work in the following context.

\begin{setup} 
\label{iteratedLC: s}
Let $S$ be a polynomial ring over $\ZZ$.  For every $t \in \NN$, fix a finite sequence $\f_t$ in $S$ generating an ideal $\ideala_t \subseteq S$.
For all $d>0$ and $\u \in \ZZ^d$, consider the compositions of functors 
\[ A^{\u} ( \ \underline{\hspace{.5cm}} \ )  = H^{u_1}( \f_1 \ ; \  \underline{\hspace{.5cm}} \ ) \circ \cdots \circ  H^{u_d}( \f_d \ ; \ \underline{\hspace{.5cm}} \  ) \text{ and } B^{\u} ( \ \underline{\hspace{.5cm}} \ ) = H^{u_1}_{\ideala_1} \circ \cdots \circ  H^{u_d}_{\ideala_d} ( \ \underline{\hspace{.5cm}} \ ) \]
from the category of $S$-modules to itself.  Note that the natural transformation from Koszul to local cohomology induces a natural transformation of functors $A^{\u} \to B^{\u}$ for every $\u \in \ZZ^d$.  
\end{setup}

\begin{remark}[Induced functors on the reductions of $S$ modulo a prime integer] Fix a prime integer $p>0$, and set $R = S/pS$.  If $M$ is an $R$-module, then $H^k_I(M)$ is also an $R$-module, and is in fact isomorphic (as $R$-modules) to $H^k_{IR}(M)$. The analogous property also holds for Koszul cohomology, and it follows from these observations that, when restricted to the category of $R$-modules, the functor $A^{\u}$ (respectively, $B^{\u}$) may be regarded as an iterated Koszul (respectively, local) cohomology functor on $R$.  It is also worth pointing out that the natural transformation from $A^{\u} \to B^{\u}$ as functors of $S$-modules restricts to a natural transformation of functors of $R$-modules, and is compatible with the one given by considering $A^{\u}$ and $B^{\u}$ as iterated Koszul and local cohomology of $R$-modules.  Finally, we note that the functor $B^{\u}$ not only induces a functor of $R$-modules, but also induces a functor of $F_R$-modules, where $F_R$ denotes the Frobenius functor on the polynomial ring $R$.
\end{remark}

The following technical result is centered on the behavior of iterated local and Koszul cohomology with respect to the sequence $0 \to S \overset{p}{\longrightarrow} S \to S/pS \to 0$, with $p>0$ a prime integer.  This result plays an important role in the proof of Theored \ref{nzd: T} later in this section, and in some sense, is what one needs to overcome the lack, in general, of long exact sequences associated to iterated cohomology.

\begin{lemma}
\label{InjectiveSurjective: L}
Let $G^{\u}$ denote either the functor $A^\u$ or $B^\u$.  Fix a positive integer $d$ and a prime integer $p>0$, and consider the exact sequence of $S$-modules $0 \to S \stackrel{p}{\longrightarrow} S \to S/pS \to 0$.    If either 
\begin{enumerate}
\item[\textup{(1)}] the induced map $G^\u(S) \overset{p}{\longrightarrow} G^{\u}(S)$ is injective for every $1 \leq c \leq d$ and $\u \in \ZZ^c$, or 
\item[\textup{(2)}] the induced map $G^\u(S) \to G^\u(S/pS)$ is surjective for every $1 \leq c \leq d$ and $\u \in \ZZ^c$, 
\end{enumerate}
then for all  $1 \leq c \leq d$ and $\u \in \ZZ^c$, we have an exact sequence
\[
0 \to G^{\u}(S) \overset{p}{\longrightarrow}  G^{\u} (S) \to  G^{\u}(S/pS) \to 0.
\]
\end{lemma}
\begin{proof}
Set $R=S/pS$.  We proceed by induction on $d$.   First, suppose that $d=1$.  In this case, the key point is that the functor $G^{u}$ with $u \in \ZZ$ has an associated long exact sequence.  Moreover, if multiplication by $p$ on $G^u(S)$ is injective for every $u \in \ZZ$, or $G^u(S) \to G^u(R)$ is surjective for every $u \in \ZZ$, then the long exact sequence in $G^{u}$ induced by \[ 0 \to S \overset{p}{\longrightarrow} S \to R \to 0 \] splits into short exact sequences, which establishes the lemma when $d=1$.

Next, suppose that the lemma is true for some positive integer $d$, and that either multiplication by $p$ is injective on $G^{\u}(S)$ for every $1 \leq c \leq d+1$ and $\u \in \ZZ^c$, or that $G^{\u}(S) \to G^{\u}(R)$ is surjective for every $1 \leq c \leq d+1$ and $\u \in \ZZ^c$. The inductive hypothesis then implies that
\begin{equation}
\label{SES: e}
0 \to G^{\u}(S) \overset{p}{\longrightarrow}  G^{\u} (S) \to  G^{\u}(R) \to 0 
\end{equation} is exact for every $1 \leq c \leq d$ and $\u \in \ZZ^c$. Thus, to complete the proof, it suffices to show that \eqref{SES: e} is exact when $\u$ is replaced with any $\w \in \ZZ^{d+1}$.  

Fix $\u \in \ZZ^d$.
As in the base case, the long exact sequence in either Koszul or local cohomology (depending on which iterated functor $G^{\u}$ represents) induced by \eqref{SES: e}
splits into short exact sequences of the form $0 \to G^{\w}(S) \overset{p}{\longrightarrow}  G^{\w}(S) \to  G^{\w}(R) \to 0$, with $\w \in \ZZ^{d+1}$,
by our assumptions on the maps induced by $G^{\w}$.
As every sequence of this form with $\w \in \ZZ^{d+1}$ arises in this way, we can conclude the proof.
\end{proof}

The following is an extension of Proposition \ref{generating: P} to the context of iterated cohomology.

\begin{proposition}  
\label{generatingII: P}  Let $R = S/pS$ denote the reduction of $S$ modulo some prime integer $p>0$.    Fix an $F_R$-module $\cM \cong F_R(\cM)$ and $\u \in \ZZ^d$, and let $\beta: A^{\u}(\cM) \to F_R(A^{\u}(\cM))$ be the map obtained by iterating\footnote{In this iterative process, we repeatedly apply the fact that the Frobenius functor is exact.} the one induced by \eqref{koszulTriangle: e}.  If $\cN$ is the $F_R$-module generated by $\beta$, then there exists an isomorphism of $F_R$-modules $\cN \cong B^{\u}(\cM)$ such that \[ \xymatrix@R=.8cm@C=1cm
{ A^{\u}(\cM)  \ar[r]^{\cong}  \ar[d] & A^{\u}(\cM) \ar[d] \\ \cN \ar[r]^-{\cong} & B^{\u}(\cM) }\]
commutes, where  $A^{\u}(\cM) \to \cN$ is the map into the direct limit, $A^{\u}(\cM) \cong A^{\u}(\cM)$ is obtained by applying $A^{\u}$ to the identification of $\cM$ with the $F$-module generated by its structure morphism, and the right-most vertical map is the one given by the natural transformation from $A^{\u}$ to $B^{\u}$.
\end{proposition}

\begin{proof}  This follows from a straightforward induction on $d \geq 1$ in which one repeatedly invokes Proposition \ref{generating: P}, with $\alpha$ being the structure morphism $\cM \cong F(\cM)$.
\end{proof}  
  
\subsection{$F$-modules and $D$-modules}  \label{FmodDmod: SS}  Let $T$ be a a ring, and let $D(T)$ denote the ring of $\ZZ$-linear differential operators on $T$.  Throughout this discussion, the term ``$D$-module''  always refer to a left $D(T)$-module.  Note that a $D$-module is always $T$-module via restriction of scalars.  

The maps in the \v{C}ech complex of a sequence in $T$ on a $D$-module $M$ are $D$-linear, and so the local cohomology of $M$ inherits a natural $D$-module structure from $M$ \cite[Example 2.1(iv)]{LyuDMod}.  This also shows that given a map of $D$-modules, the associated (functorial) maps on local cohomology are again $D$-linear.  It follows that the iterated local cohomology modules of a $D$-module are themselves $D$-modules, and that the functorial maps determined by iterated local cohomology are $D$-linear.  There is another canonical $D$-module structure on iterated local cohomology in positive characteristic.  Indeed, if $T$ has characteristic $p>0$ with Frobenius functor $F$, then 
the structure morphism $M \cong F(M)$ of an $F$-module $M$ allows one to define a natural $D(T)$-module structure on $M$ \cite[Section 5]{LyuFmod}.  Fortunately, these two $D$-module structures on (iterated) local cohomology (one coming from the \v{C}ech complex, and the other induced by the $F$-module structure) are isomorphic (see, e.g., \cite[Example 5.1(b) and 5.2(c)]{LyuFmod}).  When referring to (iterated) local cohomology in positive characteristic as a $D$-module, we  always mean either of these two isomorphic $D$-module structures.

The following result gives an important criterion for when a subset of given $F$-module generates it as $D$-module, and plays a crucial role in our proof of Proposition \ref{surjectionImpliesSurjection: P}.

\begin{lemma}[\textup{\cite[Corollary 4.4]{AMBL}}]
\label{AMBL: L}
Suppose that $T$ is a regular finitely generated algebra over an $F$-finite regular local ring of characteristic $p>0$.  If $M$ is a finitely generated $T$-module, and a $T$-linear map $M \to F(M)$ generates the $F$-module $\cM$, then the image of $M$ in the direct limit $\cM$ generates $\cM$ as a $D(T)$-module.
\end{lemma}

We now specialize to the context of Setup \ref{iteratedLC: s}. 

\begin{proposition}
\label{surjectionImpliesSurjection: P}
Fix  $\u \in \ZZ^d$ and a prime integer $p>0$.  If the map  $A^{\u}(S) \to A^{\u}(S/pS)$ induced by the map $S \to S/pS$ is surjective, then so is the induced map $B^{\u}(S) \to B^{\u}(S/pS)$.
\end{proposition}

\begin{proof}  Set $R = S/pS$. Let $\alpha : R \cong F(R)$ be defined via $x \mapsto x \otimes 1$, and let $\beta: A^{\u}(R) \to F(A^{\u}(R))$ be the map given by iterating the one induced by \ref{koszulTriangle: e}.  If $\cN$ is the $F$-module generated by $\beta$, then by Proposition \ref{generatingII: P}, there exists an isomorphism of $F$-modules $\cN \cong B^{\u}(R)$ such that
\[ \xymatrix@R=.8cm@C=1cm
{ A^{\u}(R)  \ar[r]^{\cong}  \ar[d] & A^{\u}(R) \ar[d]^{\phi} \\ \cN \ar[r]^-{\cong} & B^{\u}(R) }\]
commutes, where $\phi$ denotes the map induced by the natural transformation from $A^{\u}$ to $B^{\u}$.  By Lemma \ref{AMBL: L}, the image of $A^{\u}(R)$ in $\cN$ generates $\cN$ as an $D(R)$-module.  However, as the lower row is an isomorphism of $F$-modules, our earlier discussion implies that it is also an isomorphism of $D(R)$-modules, and consequently, the image of $A^{\u}(R)$ under the composition 
\[ A^{\u}(R) \to \cN \cong B^{\u}(R), \] which equals $\IM(\phi)$ by the commutativity of the diagram, generates $B^{\u}(R)$ as a $D(R)$-module.

Next, consider the following commutative diagram, which is obtained by applying the natural transformation $A^{\u} \to B^{\u}$ to the canonical surjection $S \to R$.
\begin{equation}
\label{NaturalCommutativeDiagram: e}
\xymatrix@R=.8cm@C=1cm{ A^{\u}(S)  \ar[r]  \ar[d] & A^{\u}(R) \ar[d]^{\phi} \\ B^{\u}(S) \ar[r]^{\psi} & B^{\u}(R)} 
\end{equation}
 
By our earlier discussion, the iterated local cohomology modules $B^{\u}(S)$ and $B^{\u}(R)$ are $D(S)$-modules, and the map $\psi$ is $D(S)$-linear.  It follows that $\IM(\psi)$ is a $D(S)$-module that is killed by $p$, and is therefore a $D(S) / p D(S) \cong D(R)$-submodule of $B^{\u}(R)$ (the preceding isomorphism follows from the discussion in \cite[Subsection 2.1]{Todos}).

We now combine these observations to complete the proof: By hypothesis, the top row in \eqref{NaturalCommutativeDiagram: e} is surjective, and therefore, $\IM(\phi) \subseteq \IM(\psi)$. However, as noted above, $\IM(\phi)$ generates $B^{\u}(R)$ over $D(R)$, and $\IM(\psi)$ is a $D(R)$-submodule of $B^{\u}(A)$, so that  
\[ B^{\u}(R) = D(R) \cdot \IM(\phi) \subseteq D(R) \cdot \IM(\psi) = \IM(\psi),\] which allows us to conclude that $\psi$ is surjective.
\end{proof}

\begin{theorem}  
\label{nzd: T}
Fix a positive integer $d$.  If multiplication by a prime integer $p>0$ is injective on $A^{\u}(S)$ for every $1 \leq c \leq d$ and $\u \in \ZZ^c$, then it is also injective on $B^{\u}(S)$ for every $1 \leq c \leq d$ and $\u \in \ZZ^c$.
\end{theorem}

\begin{proof}  Set $R=S/pS$.  According to Lemma \ref{InjectiveSurjective: L}, our hypothesis that multiplication by $p$ is injective implies that $A^{\u}(S) \to A^{\u}(R)$ is surjective for every $1 \leq c \leq d$ and $\u \in \ZZ^c$.  It follows from Proposition \ref{surjectionImpliesSurjection: P} that $B^{\u}(S) \to B^{\u}(R)$ is also surjective for every $1 \leq c \leq d$ and $\u \in \ZZ^c$, and applying Lemma \ref{InjectiveSurjective: L} once more shows that multiplication by $p$ on $B^{\u}(S)$ is injective for every $1 \leq c \leq d$ and $\u \in \ZZ^c$.
\end{proof}

\begin{corollary}
\label{finitelyManyBadPrimes: C}
Given a positive integer $d$, there are only finitely many prime integers $p>0$ for which multiplication by $p$ on is not injective on $B^{\u}(S)$ for some $\u \in \ZZ^d$.
\end{corollary}

\begin{proof}
If $p$ is a zero divisor on $B^{\u}(S)$ for some $\u \in \ZZ^d$, then it then follows from Theorem \ref{nzd: T} that $p$ must be a zero divisor $A^{\w}(S)$ for some $1 \leq c \leq d$ and $\w \in \ZZ^c$, and therefore that $p$ is contained in an associated prime of such a module.  The fact that there are only finitely many such $p>0$ then follows from the fact that there are only finitely many non-zero iterated Koszul cohomology modules of this form, and that each associated prime of each such module (of which there are only finitely many, since these modules are finitely generated over $S$) can contain at most one positive prime integer. 
\end{proof}

\subsection{The finiteness of associated primes of iterated local cohomology} \label{finitenessAss: ss}
We are now ready to prove Theorem \ref{Thm Ass} from the introduction.  In the context of Setup \ref{iteratedLC: s}, this theorem states that for every positive integer $d$ and $\u \in \ZZ^d$, the module $B^{\u}(S)$ has finitely many associated primes.

\begin{proof}[Proof of Theorem \ref{Thm Ass}]  Fix a positive integer $d$ and an element $\u \in \ZZ^d$, and set $B = B^{\u}(S)$.  The unique map $\Spec(S) \to \Spec(\ZZ)$ induces a map $\pi: \Ass_S(B) \to \Spec(\ZZ)$.  To show that $\Ass_S(B)$ is finite, it suffices to show that the image of $\pi$ is finite, and that the fiber over any point in the image of $\pi$ is also finite.

To show that the image of $\pi$ is finite, it suffices to show that it contains only finitely many positive prime integers.  However, if a prime integer $p>0$ is contained in the image of $\pi$, then $p$ is contained in some associated prime of $B$, and is therefore a zero divisor on $B$.  Corollary \ref{finitelyManyBadPrimes: C} then shows that there are only finitely many such prime integers.  

We now consider the fibers of $\pi$.  The associated primes in $\pi^{-1}(0)$ are in one-to-one correspondence with the associated primes of the localization $\QQ \otimes_{\ZZ} B = B^{\u}(\QQ \otimes_{\ZZ} S)$, and there are only finitely many such primes by \cite[Remark 3.7(i)]{LyuDMod}.  On the other hand, the primes in $\pi^{-1}(p)$ for some prime integer $p>0$ are precisely the associated primes of $B$ containing $p$, and there are only finitely many such primes by \cite[Theorem 1.2]{NunezPR}. 
\end{proof}

\section{On the agreement of (standard) and mixed characteristic Lyubeznik numbers} \label{Ln&id}


In this section, we focus on Question \ref{Q InjDim} from the introduction, which is concerned with the equality of the (standard) and mixed characteristic Lyubeznik numbers in a natural context in which they are both defined (namely, for local rings of characteristic $p>0$ obtained from some fixed finitely generated $\ZZ$-algebra).  The first subsection is dedicated to establishing some important results (the results established in this subsection are utilized in both the current and the next section).  In the second subsection, we prove Theorem \ref{Thm Lyubeznik Numbers} (in fact, we prove a more precise statement in Theorem \ref{Thm Lyubeznik Numbers Refined}).  We refer the reader to the introduction for the definition of the standard and mixed characteristic Lyubeznik numbers.

\subsection{Preliminary lemmas}

Throughout this subsection, we suppose that $T$ is an unramified regular local ring of mixed characteristic $p>0$, with residue field $\kk$.   In what follows, we  repeatedly use (without further reference) the fact  that local cohomology modules of $T$ and $T/pT$ with support in arbitrary ideals have finite Bass numbers \cite{LyuUMC, Huneke}.


\begin{lemma}
\label{Comp Ext}
If $M$ is a $T$-module such that $pM = 0$, then 
\[ \Ext^i_T(\kk, M) \cong \Ext^i_{T/pT}(\kk, M) \oplus \Ext^{i-1}_{T/pT} (\kk, M) \] as modules over $T/pT$ for every integer $i$.
\end{lemma}


\begin{proof}  Let $\boldsymbol{g}$ be a finite sequence in $T$ such that, together with $p$, forms a regular system of parameters for $T$.  By definition, the complex $\cK^{\bullet}(p, \boldsymbol{g}; M)$ equals 
\[ \cK^\bullet(\boldsymbol{g}; T) \otimes_T  \cK^\bullet(p; M)  =  \cK^\bullet(\boldsymbol{g}; T) \otimes_T \left( 0 \to M \overset{p}{\longrightarrow} M \to 0 \right),\]
and therefore, 
\[ \cK^i(p, \boldsymbol{g}; M) = \left( \cK^i(\boldsymbol{g}; T) \otimes_T M \right) \oplus \left( \cK^{i-1}(\boldsymbol{g}; T) \otimes_T M \right), \]
where the copy of $M$ in the left-hand summand is in degree zero, and that in the right-hand summand is in degree one.  By hypothesis, multiplication by $p$ on $M$ is the zero map, and given this, it follows that the differentials on $\cK^{\bullet}(p, \boldsymbol{g}; M)$ preserve each summand in this decomposition.  In fact, it is straightforward to verify that, up to a sign, the induced maps on each summand agree with the differentials on the complexes $\cK^i(\boldsymbol{g}; M)$ and $\cK^{i-1}(\boldsymbol{g}; M)$, respectively.  In other words, there is an isomorphism of complexes
\begin{equation} \label{E: Koszul}
\cK^i(p, \boldsymbol{g}; M) \cong  \cK^i(\boldsymbol{g}; M) \oplus \cK^{i-1}(\boldsymbol{g}; M).
\end{equation}

By our choice of $\boldsymbol{g}$, we may compute $\Ext^i_{T} (\kk,M)$ as the cohomology of $\cK^i(p, \boldsymbol{g}; M)$. Moreover, the image of $\boldsymbol{g}$ modulo $p$ forms a system of parameters for $T/pT$, and the Koszul complex of $\boldsymbol{g}$ modulo $p$ on $M$, considered as a module over $T/pT$, can be identified with the complex $\cK^\bullet(\boldsymbol{g}; M)$.  As before, this complex can be used to compute the modules $\Ext_{T/pT}^i(\kk, M)$, and the lemma then follows from \eqref{E: Koszul}. 
\end{proof}

\seg{Using Lemma \ref{Comp Ext}, we can relate the Bass numbers of certain local cohomology modules of $T$ with those of certain local cohomology modules of the reduction of $T$ modulo $p$.}


\begin{corollary}\label{Ext nzd}
If $\ideala$ is an ideal of $T$ such that multiplication by $p$ is injective on $H^j_{\ideala}(T)$ for every $j \geq 0$, then for every pair of nonnegative integers $i$ and $j$,
\[ \dim_{\kk} \Ext^i_T (\kk,H^j_\ideala(T)) =\dim_{\kk} \Ext^{i-1}_{T/pT} (\kk,H^j_\ideala(T/pT)).\]
\end{corollary}
\begin{proof}
Our hypothesis on $p$ guarantees that the long exact sequence in local cohomology with respect to $\ideala$
associated to the short exact sequence 
$0\to T \overset{p}{\longrightarrow} T\to T/pT\to 0$ breaks into short exact sequences 
\[ 0\to H^j_\ideala(T)\overset{p}{\longrightarrow} H^j_\ideala(T)\to H^j_\ideala(T/pT)\to 0.\]
In turn, each such short exact sequence in local cohomology induces a long exact sequence
\[\xymatrix@C=0.4cm@R=0.6cm{
\cdots \ar[r] &\Ext^{i-1}_{T}(\kk, H^j_\ideala(T)) \ar[r]^p &\Ext^{i-1}_{T}(\kk,H^j_\ideala(T)) \ar[r] &\Ext^{i-1}_{T}(\kk,H^j_\ideala(T/pT))  \connarrow & \\
& \Ext^{i}_{T}(\kk,H^j_\ideala(T)) \ar[r]^p & \Ext^{i}_{T}(\kk,H^j_\ideala(T)) \ar[r] & \cdots, \ \ \ \ \  &
}\]
and as multiplication by $p$ is zero on each $\Ext^i_{T} (\kk,H^j_\ideala(T) )$, we obtain short exact sequences 
\begin{equation}
\label{interestingExtSequence: e}
0 \to  \Ext^{i-1}_{T}(\kk,H^j_\ideala(T)) \to \Ext^{i-1}_{T}(\kk,H^j_\ideala(T/pT)) \to \Ext^{i}_{T}(\kk,H^j_\ideala(T)) \to 0 
\end{equation}
for every pair of integers $i$ and $j$.

We are now ready prove the corollary by induction on $i \geq 0$.   If $i=0$, then it is automatic that $\Ext^{i-1}_{T/pT} (\kk,H^j_\ideala(T/pT))$ is zero, while the fact that $\Ext^i_T (\kk,H^j_\ideala(T)) $ is also zero follows from \eqref{interestingExtSequence: e}.  Next, suppose that 
\[ \dim_{\kk} \Ext^i_T (\kk,H^j_\ideala(T)) =\dim_{\kk} \Ext^{i-1}_{T/pT} (\kk,H^j_\ideala(T/pT)) \]
for some $i \geq 0$ and all $j \geq 0$.  According to \eqref{interestingExtSequence: e}, 
\[ \dim_{\kk} \Ext^{i+1}_T(\kk, H^j_\ideala(T)) = \dim_{\kk} \Ext^i_{T}(\kk,H^j_\ideala(T/pT)) - \dim_{\kk} \Ext^i_{T}(\kk,H^j_\ideala(T)), \] and by Lemma \ref{Comp Ext} (with $M=H^j_\ideala(T/pT)$), we also have that
\[ \dim_{\kk} \Ext^i_{T}(\kk,H^j_\ideala(T/pT)) =  \dim_{\kk} \Ext^i_{T/pT}(\kk, H^j_\ideala(T/pT) ) + \dim_{\kk} \Ext^{i-1}_{T/pT} (\kk, H^j_\ideala(T/pT) ). \] 
Substituting the second identity into the first shows that $\dim_{\kk} \Ext^{i+1}_T (\kk,H^j_\ideala(T))$ equals
\begin{align*}
\dim_{\kk} \Ext^i_{T/pT}(\kk, H^j_\ideala(T/pT) ) + \dim_{\kk} \Ext^{i-1}_{T/pT} (\kk, H^j_\ideala(T/pT) ) -  \dim_{\kk} \Ext^i_{T}(\kk,H^j_\ideala(T)),
\end{align*}
and our inductive hypothesis implies that the two right-most terms are equal, leaving 
\[ \dim_{\kk} \Ext^{i+1}_T (\kk,H^j_\ideala(T)) = \dim_{\kk} \Ext^i_{T/pT}(\kk, H^j_\ideala(T/pT) ), \]
which allows us to conclude the proof.
\end{proof}

\seg{We now shift our attention to the study of certain local cohomology modules with support in ideals containing a prime integer $p>0$.}


\begin{lemma} 
\label{isom+SES: L}  
If $\idealb$ is an ideal of $T$ containing $p$, then $H^{j-1}_{\idealb}(T_p/T) \cong H^{j}_{\idealb}(T)$ for every $j \geq 0$.  Furthermore, if $\ideala$ is an ideal of $T$ such that multiplication by $p$ is injective on $H^j_{\ideala}(T)$ for every $j \geq 0$, and $\idealb = \ideala + pT$, then for every $j \geq 0$, there is a short exact sequence 
\[
0\to H^j_{\ideala}(T/pT)\to H^j_{\idealb}(T_p/T) \overset{p}{\longrightarrow} H^j_{\idealb}(T_p/T)\to 0.
\]
\end{lemma}


\begin{proof}
Since $\idealb T_p = T_p$, the long exact sequence in local cohomology induced by the short exact sequence $0\to T\to T_p\to T_p/T\to 0$ implies that $H^{j-1}_{\idealb}(T_p/T) \cong H^{j}_{\idealb}(T) \text{ for all integers $j \geq 0$},$
establishing the first claim. 

Next, let $\ideala$ be above, and set $\idealb = \ideala + pT$.  Since multiplication by $p$ is injective on $H^j_\ideala(T)$, the localization map $H^j_\ideala(T) \to H^j_\ideala(T_p)$ is injective, and therefore, the long exact sequence 
\[ \cdots \to H^j_\ideala(T) \to H^j_\ideala(T_p) \to H^{j+1}_{\idealb}(T)\to  \cdots \] breaks into short exact sequences
\begin{equation}
 \label{expressingAsQuotient: e}
 0 \to H^j_\ideala(T) \to H^j_\ideala(T_p) \to H^{j+1}_{\idealb}(T)\to 0.
 \end{equation}
 
As $p$ is a unit in $T_p$, multiplication by $p$ on $H^j_\ideala(T_p) \cong H^j_\ideala(T) \otimes_T T_p$ is surjective, and \eqref{expressingAsQuotient: e} then implies that multiplication by $p$ must also be surjective on the quotient  $H^{j+1}_{\idealb}(T)$.   It then follows from the first part of the lemma that multiplication by $p$ is surjective on $H^j_{\idealb}(T_p/T)$, and consequently, the long exact sequence in local cohomology associated to the short exact sequence \[ 0\to T/pT\to T_p/T \overset{p}{\longrightarrow} T_p/T \to 0,\] in which the first homomorphism sends the class of $x$ to the class of $\frac{x}{p}$, induces short exact sequences
$0\to H^j_{\idealb}(T/pT)\to H^j_{\idealb}(T_p/T) \overset{p}{\longrightarrow} H^j_{\idealb}(T_p/T)\to 0$ for every $j \geq 0$.  Finally, because the expansions of $\ideala$ and $\idealb$ to $T/pT$ agree, we may replace $\idealb$ with $\ideala$ in the first module in this short exact sequence.
\end{proof}


\begin{corollary}
\label{Ext2}
If $\ideala$ is an ideal of $T$ such that multiplication by $p$ is injective on $H^j_{\ideala}(T)$ for every $j \geq 0$, and $\idealb = \ideala + pT$, then for every pair of nonnegative integers $i$ and $j$,
\[ \dim_{\kk} \Ext^i_{T}(\kk,H^{j}_{\idealb}(T)) =\dim_{\kk} \Ext^i_{T/pT}(\kk,H^{j-1}_{\ideala}(T/pT)).\] 
\end{corollary}

\begin{proof}
The short exact sequence given in Lemma \ref{isom+SES: L} induces 
the long exact sequence 
\[\xymatrix@R=0.5cm@C=.3cm{
\cdots \ar[r] &\Ext^{i-1}_{T}(\kk, H^j_{\ideala}(T/pT)) \ar[r] &\Ext^{i-1}_{T}(\kk,H^j_{\idealb}(T_p/T)) \ar[r]^p &\Ext^{i-1}_{T}(\kk,H^j_{\idealb}(T_p/T))  \connarrow & \\
& \Ext^i_{T}(\kk,H^j_{\ideala}(T/pT)) \ar[r] &\Ext^i_{T}(\kk,H^j_{\idealb}(T_p/T)) \ar[r]^p &\Ext^i_{T}(\kk,H^j_{\idealb}(T_p/T)) \ar[r]  & \cdots. 
}\]
As multiplication by $p$ is zero on $\kk$,  it is also zero on every module in this sequence, which therefore breaks into short exact sequences 
\[ 0 \to \Ext_T^{i-1}(\kk, H^{j}_{\idealb}(T_p/T)) \to \Ext_T^i(\kk, H^j_{\ideala}(T/pT)) \to \Ext_T^{i}(\kk, H^{j}_{\idealb}(T_p/T)) \to 0. \] 

On the other hand, setting  $M= H^j_{\ideala}(T/pT)$ in Lemma \ref{Comp Ext} also shows that \[ \Ext^i_{T}(\kk,H^j_{\ideala}(T/pT)) \cong \Ext^i_{T/pT}(\kk,H^j_{\ideala}(T/pT)) \oplus \Ext^{i-1}_{T/pT}(\kk,H^j_{\ideala}(T/pT)).  \]

Comparing these two descriptions of $\Ext^i_{T}(\kk,H^j_{\idealb}(T/pT))$, and then inducing on $i \geq 0$, show that $\dim_{\kk} \Ext_{T/pT}^i ( \kk, H^{j-1}_{\ideala}(T/pT) ) = \dim_{\kk} \Ext_T^i(\kk, H^{j-1}_{\idealb}(T_p/T))$ for all $i, j \geq 0$, and the corollary then follows from the first isomorphism in Lemma \ref{isom+SES: L}.
\end{proof}

\subsection{On the agreement of Lyubeznik numbers}
Theorem \ref{Thm Lyubeznik Numbers} is a statement about the Lyubeznik numbers of rings obtained from a polynomial ring over $\ZZ$ by first killing a prime integer $p>0$, and then localizing at a prime ideal.  Below, we set the notation needed to precisely describe this process (and this notation  also appear in Theorem \ref{Thm Lyubeznik Numbers Refined}).

Let $S$ denote a polynomial ring over $\ZZ$, fix an ideal $I$ of $S$, and set $R=S/I$.  Furthermore, fix a prime integer $p>0$, and an ideal $Q$ of $S$ containing a prime integer $p>0$ and $I$.  The ideal $Q$ expands to a prime ideal under the canonical map from $S$ to each of the rings $S/pS, R$, and $R/pR = S/(I+pS)$; abusing notation, we also use $Q$ to denote the expansion of $Q \subseteq S$ to these rings.  We also set $T = S_{Q}$, so that $(T, QT, \kk)$ is a regular local ring of mixed characteristic $p$.   In fact, because 
$T/pT$ is the localization of a polynomial ring over $\mathbb{F}_p$ (and, in particular, a regular local ring), the prime $p$ must be part of a minimal generating set for the maximal ideal of $T$ (i.e., $T$ is unramified).  Finally, we set \[ A = (R/pR)_Q = S_Q / (IS_Q+pS_Q) = T / (IT + pT).\] 

To prove Theorem \ref{Thm Lyubeznik Numbers}, it suffices to show that there exists a finite set of prime integers $W$ for which $\lambda_{i,j}(A) = \newLyu_{i,j}(A)$ for $i,j \geq 0$ whenever $p \notin W$.  Below, we establish a slightly more refined statement.

\begin{theorem}[cf.\,Theorem \ref{Thm Lyubeznik Numbers}]
\label{Thm Lyubeznik Numbers Refined}
Under the above notation, let $W$ denote set of prime integers that are zero divisors on $H^j_I(S)$ for some $j \geq 0$ \textup{(}which is a finite set by Corollary \ref{finitelyManyBadPrimes: C}\textup{)}.  
If $p \notin W$ and $i,j \geq 0$, then 
\[ \lambda_{i,j}(A) = \newLyu_{i,j}(A) = \newLyu_{i+1,j+1}(R_Q). \] 
\end{theorem}

\begin{proof}

By definition, we have an exact sequence 
\[ 0 \to IT \to T \to R_Q \to 0 \ \text{ and } \ 0 \to IT + pT \to T \to A \to 0 \] of $T$-modules, and an exact sequence of $T/pT$-modules
\[ 0 \to (IT+pT)/pT \to T/pT \to A \to 0. \]

Set $d = \dim(T)$.  As the left-most term in the short exact sequence of $T/pT$-modules agrees with the expansion of $IT$ to $T/pT$, these sequences imply that 
\begin{align*}
\newLyu_{i+1,j+1}(R_Q)  &= \dim_{\kk} \Ext^{i+1}_T(\kk, H_{IT}^{d - j-1}(T)), \\
\newLyu_{i,j}(A)  &= \dim_{\kk} \Ext^i_{T}(\kk,H^{d -j}_{IT+pT}(T)), \text{ and } \\
\lambda_{i,j}(A) & = \dim_{\kk} \Ext^i_{T/pT}(\kk, H_{IT}^{d-j-1}(T/pT)).
\end{align*}

As $p \notin W$, multiplication by $p$ is injective on $H^j_I(S)$ for every $j \geq 0$.  As $T$ is flat over $S$, it follows that multiplication by $p$ is also injective on $H^j_{IT}(T) = T \otimes_S H^j_I(S) $ for every $j \geq 0$.   In light of this, we may then combine Corollary \ref{Ext nzd} (with $\ideala = IT$) and our description of the Lyubeznik numbers above to conclude that $\newLyu_{i+1, j+1}(R_Q) = \lambda_{i,j}(A)$.  Similarly, we may apply Corollary \ref{Ext2} (with $\ideala = IT$) to concluce that $\newLyu_{i,j}(A)=\lambda_{i,j}(A)$.
\end{proof}

\section{Affirmative answers to Lyubeznik's question} \label{positiveInjDim: s}
In this section, we consider Lyubezink's question (i.e.,  Question \ref{Q InjDim} from the introduction), which asks the following:  Suppose $I$ is an ideal of a regular local ring $(S, \m)$ such that $H^j_I(S) \neq 0$.
\begin{itemize}
\item Is $\InjDim_S H^j_I(S)$ less than or equal to $\dim \Supp_S H^j_I(S)$?
\item Is the iterated local cohomology module $H^i_{\m} H^j_I(S)$ injective for every $i \geq 0$?
\end{itemize}

In the first subsection (namely, in Theorem \ref{positiveInjDim: T}), we give a positive answer to the first question in an interesting case, and in the third subsection (namely, in Proposition \ref{ILC: p}), we do the same for the second question.  In the second subsection, we discuss the condition that local cohomology over a finitely generated $\ZZ$-algebra descends from local cohomology over $\QQ$.


\subsection{Positive answers to the question on injective dimension}

\begin{remark}[On the support of certain local cohomology]  Suppose that $\ideala$ is an ideal of a ring $T$.  If $j \geq 0$ and $p>0$, then because multiplication by $p$ is zero on $H^j_{\ideala}(T/pT)$, 
\begin{equation}
\label{SameSuppOverQuotient: e}
 \Supp_T H^j_{\ideala}(T/pT)  = \Supp_T H^j_{\ideala}(T/pT) \cap \mathbb{V}(pT)  = \Supp_{T/pT}  H^j_{\ideala}(T/pT),
\end{equation}
where we have identified $\Spec (T/pT)$ with $\mathbb{V}(pT) \subseteq \Spec T$.     If, in addition, multiplication by $p$ is injective on every local cohomology module of $T$ with support in $\ideala$, then the exact sequence
\[ 0 \to H_{\ideala}^j(T) \overset{p}{\longrightarrow}  H_{\ideala}^j(T) \to H_{\ideala}^j(T/pT) \to 0 \] implies further that 
\[ \Supp_T H^j_{\ideala} (T/pT)  \subseteq ( \Supp_T H^j_{\ideala}(T) ) \cap \mathbb{V}(pT).\]  

Furthermore, because every prime in $(\Supp_T H^j_{\ideala}(T)) \cap \mathbb{V}(pT)$ must properly contain an associated prime of $H^J_{\ideala}(T)$ (the properness follows from the assumption that $p$ is a nonzerodivisor on $H^j_{\ideala}(T)$, and therefore cannot be contained in any associated prime), this shows that 
\begin{equation} 
\label{boundOnDimSupp: e}
\dim \Supp_T H^j_{\ideala} (T/pT) \leq \dim ( \Supp_T H^j_{\ideala}(T) \cap \mathbb{V}(pT) ) \leq \dim \Supp_T H^j_{\ideala}(T) - 1.
\end{equation} 
\end{remark}

\begin{lemma} \label{vanishingBass: L}
Let $(T, \m)$ be an unramified regular local ring of mixed characteristic $p > 0$. If $Q \subsetneq \m$ is a prime ideal of $T$, and $\ideala$ is an arbitrary ideal of $T$, then for $j\geq0$,
\[
\Ext^i_{T_Q} (T_Q/Q T_Q,H^j_{\ideala}(T_Q))=0 \ \text{ for } \ i>\dim\Supp_T H^j_{\ideala}(T).
\]
\end{lemma}


\begin{proof}
Any chain of prime ideals in $\Supp_T H^j_{\ideala}(T)$ contained in $Q$ can be extended by adding $\m$. In other words, $\dim\Supp_T H^j_{\ideala}(T_Q ) \leq \dim\Supp_T H^j_{\ideala}(T)-1$, so that by \cite[Theorem 5.1]{Zhou},
\[\InjDim H^j_{\ideala}(T_Q )\leq \dim\Supp_T H^j_{\ideala}(T_Q )+1\leq   \dim\Supp_T H^j_{\ideala}(T). \] 
Thus, all Bass numbers of $H^j_{\ideala}(T_Q)$ vanish above $\dim\Supp_T H^j_{\ideala}(T)$.
\end{proof}

\begin{proposition}  \label{injDimBoundp: p}
Let $T$ be an unramified regular local ring of mixed characteristic $p > 0$, and let $\ideala$ be an ideal of $T$ such that multiplication by $p$ is injective on every local cohomology module of $T$ with support in $\ideala$.  
If $\idealb = \ideala+pT$ and $H^j_{\idealb}(T)$ is nonzero, then 
\[\InjDim_T H^j_{\idealb}(T)\leq \dim\Supp_{T} H^j_{\idealb}(T).\]
\end{proposition}

\begin{proof}
Given such a nonzero local cohomology module $H^j_{\idealb}(T)$,
we aim to show that all of its Bass numbers beyond the dimension of its support are zero.
Lemma \ref{vanishingBass: L} shows that this holds for Bass numbers with respect to non-maximal prime ideals, so it remains to show that 
\begin{equation} \label{goalBassMax: e}
\dim_{\kk} \Ext^i_{T}(\kk , H^j_{\idealb}(T)) = 0 \ \text{ for } \ i > \dim \Supp_T H^j_{\idealb}(T),
\end{equation}
where $\kk$ denotes the residue field of $T$.

To establish \eqref{goalBassMax: e}, our strategy will be to ``replace" the $T$-modules in this statement with $T/pT$-modules, and then appeal to results in \cite{LyuDMod} (the key point being that $T/pT$ is a regular local ring of characteristic $p>0$).  As a first step in this direction, observe that Corollary \ref{Ext2}
 allows us to restate \eqref{goalBassMax: e} as
\[ \dim_{\kk} \Ext^i_{T/pT} (\kk,H^{j-1}_{\idealb}(T/pT))
 = 0 \ \text{ for } \ i > \dim \Supp_T H^j_{\idealb}(T).
\]

Assume, momentarily, that 
\begin{equation}
\label{SameSupp: e}
\Supp_T H^j_{\idealb}(T) = \Supp_{T/pT} H^{j-1}_{\ideala}(T/pT), 
\end{equation}
where we regard the latter set as a subset of $\Spec (T/pT) = \mathbb{V}(pT) \subseteq \Spec T$.  Given this, we may rewrite \eqref{goalBassMax: e} as the equivalent condition 
\[ \dim_{\kk} \Ext^i_{T/pT} (\kk,H^{j-1}_{\idealb}(T/pT))
 = 0 \ \text{ for } \ i > \dim \Supp_{T/pT} H^{j-1}_{\ideala}(T/pT),
\]
and this condition holds by \cite[Corollary 3.6]{LyuDMod}.

Thus, to conclude the proof, it suffices to justify \eqref{SameSupp: e}.  Towards this,  notice that localizing the sequence given by Lemma \ref{isom+SES: L} at a prime ideal $Q $ of $T$ gives a short exact sequence
\[0\to H^{j-1}_{\ideala}(T/pT)_Q  \to H^{j-1}_{\idealb}(T_p/T)_Q  \FDer{p} H^{j-1}_{\idealb}(T_p/T)_Q  \to 0.\]
As $p \in \idealb$, multiplication by $p$ cannot be injective on $H^{j-1}_{\idealb}(T_p/T)_Q $ unless it is the zero module, and so 
$H^{j-1}_{\idealb}(T_p/T)_Q$ vanishes if and only if $H^{j-1}_{\ideala}(T/pT)_Q$ does.  In other words, 
\[  \Supp_T H^{j-1}_{\idealb}(T_p/T) = \Supp_T H^{j-1}_{\ideala}(T/pT). \] 
To establish \eqref{SameSupp: e}, simply note that the left-hand side above agrees with $\Supp_T H^j_{\idealb}(T)$ by Lemma \ref{isom+SES: L}, and that the right-hand side above agrees with $\Supp_{T/pT} H^{j-1}_{\ideala} (T/pT)$ by \eqref{SameSuppOverQuotient: e}.
\end{proof}


\begin{theorem}[cf.\,Theorem \ref{Thm Injective Module}] \label{positiveInjDim: T}
Let $S$ be a polynomial ring over $\ZZ$, and $I$ an ideal of $S$.
Let $W$ denote the set of prime integers that are zero divisors on $H^j_I(S)$ for some $j \geq 0$ (which is a finite set by Corollary\ \ref{finitelyManyBadPrimes: C}). If $Q \in \Supp_S H^j_I(S)$ does not lie over a prime in $W$, then 
\[\InjDim_S H^j_{I}(S)_Q \leq \dim \Supp_S H^j_{I}(S)_Q.\]
\end{theorem}

\begin{proof}
Suppose $Q$ is a prime ideal in $\Supp_S H^j_I(S)$ lying over $p \notin W$. 
Let $T=S_Q$, and let $\kk$ be its residue field.  As in the proof of Proposition \ref{injDimBoundp: p}, we reduce to proving a statement about $T/pT$-modules, and then appeal to \cite{LyuDMod}.  We stress, however, that the reduction step here is different from the one used in the proof of Proposition \ref{injDimBoundp: p}.

Since all Bass numbers of $H^j_{IT}(T)$ with respect to non-maximal prime ideals of $T$ vanish beyond the dimension of its support by Lemma \ref{vanishingBass: L}, it is enough to show that 
\[ \dim_K \Ext^i_{T}(\kk,H^j_{IT}(T)) = 0 \text{ if } i > \dim \Supp_T H^j_{IT}(T), \] which, by Corollary \ref{Ext nzd}, may be restated as 
\begin{equation}
\label{restatementVanishingBassNumbers: e}
 \dim_K\Ext^{i-1}_{T/pT}(\kk,H^j_{IT}(T/pT)) = 0 \text{ if } i > \dim \Supp_T H^j_{IT}(T). \end{equation} 

By our choice of $p$, \eqref{boundOnDimSupp: e} shows that \[ \dim \Supp_{T/pT} H^j_{IT}(T/pT) +1 \leq \dim \Supp_T H^j_{IT}(T),\] and therefore, the condition in \eqref{restatementVanishingBassNumbers: e} holds whenever 
\[  \dim_K\Ext^{i-1}_{T/pT}(\kk,H^j_I(T/pT)) = 0 \text{ if } i -1> \dim \Supp_T H^j_{IT}(T). \]  
However, this last condition holds by \cite[Corollary 3.6]{LyuDMod}.
\end{proof}

\subsection{On the condition that local cohomology of $\ZZ$-algebras descends from $\QQ$}
\label{DescendsFromQQ: ss}

Let $S$ be a polynomial ring over $\ZZ$, and $I$ an ideal of $S$.  In this subsection, we consider the condition that \[ \QQ \otimes_{\ZZ} H^j_I(S) \neq 0, \] which appears in the statement of Theorem \ref{Thm Injective Module}, but not in the statement of Theorem \ref{positiveInjDim: T}.   

First, we consider what occurs when this condition fails:  Fix an integer $j \geq 0$, and set $H = H^j_I(S)$.  If $\QQ \otimes_{\ZZ} H= 0$, then no associated prime of $H$ lies over $0 \in \ZZ$.  Consequently, as any prime $Q \in \Supp_S H$ contains an associated prime $Q_0$ of $H$, the preceding observation shows that $p \ZZ =  Q_0 \cap \ZZ \subseteq Q \cap \ZZ$ for some prime integer $p>0$, and it follows that $Q \cap \ZZ = p\ZZ$ as well.  However, being an element of an associated prime,  $p$ is a zero divisor on $H$, and therefore, $Q$ lies over a prime contained in the set $W$ appearing in the statement of Theorem \ref{positiveInjDim: T}.  It follows from this that for such local cohomology modules, Theorem \ref{positiveInjDim: T} is vacuous.   We provide a concrete example of this situation below.

\begin{example}
Let $S$ be a polynomial ring in six indeterminates over $\ZZ$,  
and let $I$ denote the monomial ideal corresponding to Reisner's variety \cite{Reisner}.  It is know that $H=H^4_I(S)$ is nonzero \cite[Example 1]{LyuMonomial}.  On the other hand, $\QQ \otimes_\ZZ S/I$ is Cohen-Macaulay \cite[Remark 3]{Reisner}, and therefore, $\QQ \otimes_\ZZ  H = 0$ by \cite[Proposition 3.1]{Montaner}.  By the preceding discussion, Theorem \ref{positiveInjDim: T} is an empty statement in this case.   See Section \ref{counterex} for a closely-related example.
\end{example}

Next, we consider the situation when local cohomology of $S$ descends from $\QQ$.  In Corollary \ref{justification: C} below, we show that this condition guarantees that Theorem \ref{positiveInjDim: T} is \emph{not} vacuous.  We start with a lemma that is well known to experts.

\begin{lemma} \label{primeUnits: L}
If $A$ is finitely generated as an algebra, and torsion-free, over $\ZZ$, then there are only finitely many prime integers that are units of $A$.
\end{lemma}

\begin{proof}
Suppose that there exists an infinite sequence of prime integers $\{ p_k \}_{k=1}^\infty$ that are units of $A$.  
By generic freeness \cite[Lemma 8.1]{HochsterRoberts}, there exists a non-zero integer $d$ such that $A_d$ is free over $\ZZ_d$.  For $k \gg 0$, the prime integer $p_k$ does not divide $d$, and  therefore, $\cap_{k=1}^\infty p_k \ZZ_d = 0$;  as $A_d$ is free over $\ZZ_d$, we also have that $\cap_{k=1}^\infty p_k A_d = 0$.  However, because multiplication by $d$ is injective on $A$, we have that $A \subseteq A_d$, and therefore, $\cap_{k=1}^\infty p_k A \subseteq \cap_{k=1}^\infty p_k A_d = 0$, which contradicts the assumption that each $p_k A = A$ for every $k \geq 1$.
\end{proof}

\begin{corollary}
\label{justification: C}  Let $A$ be a domain that contains $\ZZ$, and that is finitely generated as an algebra over $\ZZ$,  and let $I$ be an ideal of $A$.  If $\QQ \otimes_{\ZZ} H^j_I(A) \neq 0$, then for all but finitely many prime integers $p>0$, there exists a prime $Q \in \Supp_A H^j_I(A)$ such that $Q \cap \ZZ = p \ZZ$.
\end{corollary}

\begin{proof}
The hypothesis that $\QQ \otimes_\ZZ H^j_I(A) \neq 0$ implies that there exists an associated prime $Q_0$ of $H^j_I(A)$ such that $Q_0 \cap \ZZ = 0$.  Set $B = A / Q_0$.  By hypothesis, $\ZZ \subseteq A$, and as $Q_0 \cap \ZZ = 0$, we also have that $\ZZ = \ZZ / (Q_0 \cap \ZZ) \subseteq B$.  As $B$ is a domain, this shows that $B$ is a finitely generated algebra and torsion-free over $\ZZ$, and by Lemma \ref{primeUnits: L}, there exists $N >0$ such that any prime $p>N$ is not a unit of $B$.
For such a prime, $Q_0 + pA \neq A$ (otherwise, $q + pa = 1$ for some $q \in Q_0$ and $a \in A$, which would then imply that $p$ was a unit of $B$). If $Q$ is any prime ideal of $A$ containing the proper ideal $Q_0 + pA$, then $Q$ contains $p$, by construction, and $Q$ is in $\Supp_A H^j_I(A)$, since it contains $Q_0$.
\end{proof}


\subsection{Positive answers to the question on iterated local cohomology} 

In this subsection, we address Lyubeznik's second question concerning the injectivity of certain iterated local cohomology modules.  Before we start our discussion, we recall a result of the second and fourth authors (this result plays an important role in both the current and subsequent section).

\begin{remark}[A criterion for injectivity]

 \label{Rem Sub Injective}   
Let $A$ be a power series ring over a complete Noetherian discrete valuation ring $(V, pV)$ of mixed characteristic $p>0$, and let $D=D(A,V)$ denote the ring of $V$-linear differential operators on $A$.  If $H$ is any (iterated) local cohomology module of $A$, then $H$ is a $D$-module \cite[Example 2.1(iv)]{LyuDMod} and has finite Bass numbers over $A$ \cite[Theorem 5.1]{NunezPR}.  Consequently, if we assume further that $H$ is supported only at the maximal ideal of $A$, then \cite[Lemma 4.2]{NuWiMixChar} implies that $H$ is an injective $A$-module if and only if multiplication by $p$ on $H$ is surjective.  
\end{remark}

\begin{remark}[A specialized criterion for injectivity]  
\label{specializedCriterion: R}
Fix a prime integer $p>0$, and let $S$ be either a polynomial ring over $\ZZ$, or 
over the localization of $\ZZ$ at ${p\ZZ}$.
Let $A$ be the completion of $S$ at the maximal ideal $\m$ generated by $p$ and the variables in $S$. 
Note that regardless of the choice of $S$, $A$ is a power series ring over the $p$-adic integers. 
Fix an (iterated) local cohomology module $H$ of $S$, and suppose that $H$ is supported only at the maximal ideal $\m$.  This condition implies that, when regarded as an $A$-module in the natural way,  $H$ is isomorphic to an (iterated) local cohomology module of $A$.  Thus, by Remark \ref{Rem Sub Injective}, if multiplication by $p$ on $H$ is surjective, then $H$ is injective over $A$.  However, as the $S$-module $H$ is supported only at $\m$, this implies that $H$ is injective over $S$ as well.  

\end{remark}

\begin{proposition}
 \label{ILC: p}
Let $S$ be a polynomial ring over $\ZZ$, and let $\n$ be the ideal of $S$ generated by the variables.  Fix an ideal $I$ of $S$, and let $W$ denote the set of prime integers that are zero divisors on some (iterated) local cohomology module of the form $H^j_I(S)$ or $H^i_{\n} H^j_I(S)$ (which is a finite set by Corollary \ref{finitelyManyBadPrimes: C}).  If $p \notin W$ and $\m = \n + pS$, then the iterated local cohomology modules $H^i_{\m} H^j_{I}(S)$ and $H^i_{\m} H^j_{I+pS}(S)$ are injective for all $i,j$ (and, in fact, $H^i_{\m} H^j_{I+pS}(S) \cong H^{i+1}_{\m} H^{j-1}_{I}(S)$).
\end{proposition}

\begin{proof}
Consider the long exact sequence 
\begin{equation} \label{LES}
\cdots \longrightarrow H^i_{\m} H^j_{I}(S) \longrightarrow H^i_{\n} H^j_{I}(S) {\longrightarrow} H^i_{\n} H^j_{I}(S)_p {\longrightarrow} H^{i+1}_{\m} H^j_{I}(S) \longrightarrow \cdots.
\end{equation}
By our choice of $p$, each localization map in this sequence is injective, and therefore, each map 
\[ H^i_{\n} H^j_{I}(S)_p {\longrightarrow} H^{i+1}_{\m} H^j_{I}(S) \] is surjective.  Therefore, as multiplication by $p$ is surjective on $H^i_\n H^j_{I}(S)_p$, it is also surjective on the quotient $H^i_{\m} H^j_{I}(S)$.  As the iterated local cohomology module $H^i_{\m} H^j_I(S)$ is supported only at $\m$, it follows from Remark \ref{specializedCriterion: R} that  $H^i_{\m} H^j_I(S)$ is an injective $S$-module.

Next, consider the long exact sequence 
\[\cdots \to  H^{j-1}_{I}(S)\to H^{j-1}_{I}(S_p)\to H^j_{I+pS}(S) \to \cdots.\]
As above, the localization maps in this sequence are injective,  and so we obtain short exact sequences
 $0\to  H^{j-1}_{I}(S)\to H^{j-1}_{I}(S_p)\to H^j_{I+pS}(S) \to 0$, each of which induces
 \[
\cdots \to H^i_{\m}H^{j-1}_{I}(S_p)\to H^i_{\m} H^j_{I+pS}(S)\to H^{i+1}_{\m} H^{j-1}_{I}(S)\to H^{i+1}_{\m} H^{j-1}_{I}(S_p) \to \cdots.
\]
Since $p$ is contained in $\m$, all modules of the form $H^i_{\m}H^j_{I}(S_p)$ vanish.  Thus,
 $H^i_{\m} H^j_{I+pS}(S)$ is isomorphic to $H^{i+1}_{\m} H^{j-1}_{I}(S)$, the latter of which we have already established is injective.
\end{proof}

\section{A negative answer to Lyubeznik's question on injective dimension} 
\label{counterex}


In this section, we fix the following conventions.
For clarity in notation, we set $p=2$.  We also let $V$ be the localization of $\ZZ$ at $p\ZZ$, and set $S=V[x_1,\ldots,x_6]$.  
Let $\m$ denote the maximal ideal of $S$ generated by $p$ and the indeterminates, and let $I$ be the ideal generated by the following monomials (cf.\,\cite[Remark 6.11]{NuWiMixChar}):
\begin{alignat*}{20}
&{\mu}_1=x_1x_2x_3 , \ \ \ &{\mu}_2=x_1x_2x_4 , \ \ \ &{\mu}_3=x_1x_3x_5, \ \ \ &{\mu}_4=x_1x_4x_6, \ \ \ &{\mu}_5=x_1x_5x_6, \\
&{\mu}_6=x_2x_3x_6, \ \ \ &  {\mu}_7=x_2x_5x_6, \ \ \ &{\mu}_8=x_2x_4x_5, \ \ \ &{\mu}_9=x_3x_4x_5, \ \ \ &{\mu}_{10}=x_3x_4x_6.
\end{alignat*}

The goal of this section is to show that $H^4_I(S)$ can be used to provide a negative answer to Lyubeznik's question (i.e., Question \ref{Q InjDim} from the introduction).


We begin with an important preliminary observation:  Consider the long exact sequence
\begin{equation*}
\cdots \to H^3_{I+pS}(S) \to H^3_I(S)\to H^3_I(S_p)\to H^4_{I+pS}(S)\to H^4_I(S)\to H^4_I(S_p) \to \cdots.
\end{equation*}
It was shown in  \cite[Remark 6.3 and the proof of Proposition 6.10]{NuWiMixChar} that $H^j_{I+pS}(S)$ vanishes unless $j = 4$.  Furthermore, $S_p / I S_p$ is known to be Cohen-Macaulay \cite[Remark 3]{Reisner}, and as $S_p$ is a polynomial ring over $\QQ$, we may apply \cite[Proposition 3.1]{Montaner} to conclude that
\begin{equation}\label{locLCzero: e}
H^j_I(S_p)=0 \text{ if } j \neq \dim(S_p)-\dim\(S_p/IS_p\) = 3.
\end{equation}
Therefore, the long exact sequence above reduces to 
\begin{equation} \label{sequence CE}
0\to H^3_I(S)\to H^3_I(S_p)\to H^4_{I+pS}(S)\to H^4_I(S)\to 0.
\end{equation}


\begin{lemma}\label{Supp Zero}
$\Supp_S H^4_I(S) \subseteq \{ \m \}$.
\end{lemma}


\begin{proof}
For $1 \leq j \leq 7$, let $\ideala_j$ denote the ideal of $S$ generated by ${\mu}_{j+1}, \ldots, {\mu}_{10}$, and let $\idealb_j = \ideala_j S_{{\mu}_j}$.
A direct computation shows that $\idealb_1 =(x_4,x_5,x_6) S_{{\mu}_1}$, and therefore,
\begin{align} \label{LCcomp: e} 
\begin{split}
H^3_{\idealb_1}(S_{{\mu}_1})   &\cong \(H^3_{(x_4, x_5, x_6)} (S) \)_{x_1 x_2 x_3} 
\cong \( \frac{ S_{x_4 x_5 x_6} }{ S_{x_4 x_5} + S_{x_4 x_6} + S_{x_5 x_6} } \)_{x_1 x_2 x_3} \\
&\cong \frac{ S_{x_1x_2x_3x_4x_5x_6} }{  S_{x_1x_2x_3x_4x_5}+S_{x_1x_2x_3x_4x_6}+S_{x_1x_2x_3x_5x_6} . }
\end{split}
\end{align}
In particular, this module vanishes after localizing at any of $x_4, x_5$, or $x_6$, so that
\begin{equation} \label{suppH3: e}
\Supp_S H^3_{\idealb_1}(S_{{\mu}_1})  \subseteq \mathbb{V}(x_4,x_5,x_6).
\end{equation} 

The following computations imply that $H^i_{\idealb_j}(S_{\mu_j})=0$ for $i \geq 3$ and $2 \leq j \leq 7$:
\begin{alignat*}{20}
&\idealb_2 = (x_5,x_6) S_{{\mu}_2} , & \ \ \ \ &\idealb_3 = (x_4,x_6) S_{{\mu}_3} , & \ \ \ \ &\idealb_4 = (x_3,x_5) S_{{\mu}_4}  \\
&\idealb_5 = (x_2,x_3x_4) S_{{\mu}_5} , & \ \ \ &\idealb_6 = (x_4,x_5) S_{{\mu}_6}, &\ \ \ \ &\idealb_7 = (x_4) S_{{\mu}_7}.
\end{alignat*}
Then, substituting this vanishing into the long exact sequence
\begin{equation*} \label{LESFirstStep2: e}
\cdots \to H^3_{\idealb_j}(S_{{\mu}_j})\to H^4_{\ideala_{j-1} S}(S)\to H^4_{\ideala_j S}(S)\to H^4_{\idealb_j}(S_{{\mu}_j})\to \cdots,
\end{equation*}
shows $H^4_{\ideala_{j-1} S}(S) \cong H^4_{\ideala_j S}(S)$ for $2 \leq j \leq 6$, and therefore, that 
\begin{equation*} \label{H4zero: e}
H^4_{\ideala_1}(S) =  \cdots  = H^4_{\ideala_7}(S) = 0,
\end{equation*}
where the last equality holds since $\ideala_7$ is generated by the three monomials $\mu_1, \mu_2, \mu_3 \in S$.
Substituting this vanishing into the long exact sequence 
\begin{equation*}
\label{LESFirstStep1: e} \cdots \to H^3_{\ideala_1}(S) \to H^3_{\idealb_1}(S_{{\mu}_1}) \to H^4_I(S) \to H^4_{\ideala_1}(S)\to H^4_{\idealb_1}(S_{{\mu}_1})\to \cdots,
\end{equation*}
shows that $H^3_{\idealb_1}(S)$ surjects onto $H^3_{I}(S)$.  It follows from this and \eqref{suppH3: e} that
\begin{equation} \label{supp456: e}
\Supp_S H^4_I(S) \subseteq \Supp H^3_{\idealb_1}(S_{\mu_1})   \subseteq \mathbb{V}(x_4,x_5,x_6).
\end{equation}

We next construct another closed set containing  $\Supp_S H^4_I(S)$ using a similar argument.
For $3 \leq j \leq 9$, let $\idealc_j$ be the ideal of $S$ generated by $\mu_1, \ldots, \mu_{j-1}$.
Let $\ideald_j = \idealc_j \, S_{\mu_j}$, so that 
\begin{alignat*}{20}
&\ideald_4 = (x_2,x_3x_5) S_{{\mu}_4}, & \ \ \ \ &\ideald_5 = (x_3,x_4)  S_{{\mu}_5}, & \ \ \ \ &\ideald_6 = (x_1) S_{{\mu}_6}, \\
&\ideald_7 = (x_1,x_3) S_{{\mu}_7} & \ \ \ \ &\ideald_8 = (x_1,x_6) S_{{\mu}_8}, & \ \ \ \ &\ideald_9 = (x_1,x_2) S_{{\mu}_9},
\end{alignat*}
and $\ideald_{10} = (x_1,x_2,x_5) S_{\mu_{10}}$.  As before, this shows that $H^i_{\ideald_j}(S_{\mu_j}) = 0$ for $i \geq 3$ and $4 \leq j \leq 10$, and combining this with the long exact sequence
\begin{equation*} \label{LESFirstStep2: e}
\cdots \to H^3_{\ideald_j}(S_{{\mu}_j})\to H^4_{\idealc_{j+1} S}(S)\to H^4_{\idealc_j S}(S)\to H^4_{\ideald_j}(S_{{\mu}_j})\to \cdots,
\end{equation*}
shows that $ H^4_{\idealc_{10}}(S) =  \cdots  = H^4_{\idealc_4}(S)=0$, where the last equality holds since $\idealc_4$ can be generated by the three monomials $\mu_8, \mu_9, \mu_{10} \in S$.  This and the long exact sequence
\begin{equation*}
\label{LESFirstStep1: e} \cdots \to H^3_{\idealc_{10}}(S) \to H^3_{\ideald_{10}}(S_{{\mu}_{10}})  \to H^4_I(S) \to H^4_{\idealc_{10}}(S)\to H^4_{\ideald_{10}}(S_{{\mu}_{10}})\to \cdots,
\end{equation*}
imply that $H^3_{\ideald_{10}}(S_{{\mu}_{10}})$ surjects onto $H^4_I(S)$, and by isomorphisms analogous to those in \eqref{LCcomp: e}, we can conclude that  $\Supp_S H^3_{\ideald_{10}}(S_{\mu_{10}})$ is contained in $\mathbb{V}(x_1, x_2, x_5)$, so that
\begin{align} \label{supp125: e} 
\Supp_S H^4_I(S) &\subseteq \Supp_S H^3_{\ideald_{10}}(S_{{\mu}_{10}}) \subseteq \mathbb{V}(x_1,x_2,x_5).
\end{align}

Notice that after interchanging $x_1$ with $x_3$ and $x_4$ with $x_6$, the set of generators of $I$ remains the same.
Thus, applying this symmetry to \eqref{supp125: e} implies that $\Supp_S H^4_I(S)\subseteq \mathbb{V}(x_2,x_3,x_5)$ as well, and these containments and \eqref{supp456: e} allow us to conclude that
\[
\Supp_S H^4_I(S)\subseteq \mathbb{V}(x_4,x_5,x_6)\cap \mathbb{V}(x_1,x_2,x_5) \cap \mathbb{V}(x_2,x_3,x_5)=\mathbb{V}(x_1,x_2,x_3,x_4,x_5,x_6).
\]
Finally, as $H^4_I(S_p)=0$ by \eqref{locLCzero: e}, we have that
\[
\Supp_S H^4_I(S)\subseteq \mathbb{V}(pS) \cap \mathbb{V}(x_1,x_2,x_3,x_4,x_5,x_6)=\{\m\}.
\]

\vspace{-.4cm}

\end{proof}


\begin{lemma} \label{multpnotsurj}
Multiplication by $p$ on $H^4_I(S)$ is not surjective.  In particular,
$H^4_I(S)\neq 0$.
\end{lemma}
\begin{proof} 
In \cite[Corollary 6.9]{NuWiMixChar}, it was shown that multiplication by $p$ on $H^4_{I+pS}(S)$ is not surjective.  With this in mind, suppose that multiplication by $p$ on $H^4_I(S)$ is surjective.  We aim to derive a contradiction showing that this implies that $p$ must also be surjective on $H^4_{I+pS}(S)$.

Consider the following commutative diagram, whose rows are the sequence \eqref{sequence CE}:
\[
\xymatrix@C=0.6cm@R=0.6cm{
0\ar[r] & H^3_I(S)\ar[r]^{\alpha} \ar[d]^p  & H^3_I(S_p)\ar[r]^{\beta} \ar[d]^p  & H^4_{I+pS}(S)\ar[r]^{\gamma}  \ar[d]^p & H^4_I(S) \ar[r] \ar[d]^p & 0\\
0\ar[r] & H^3_I(S)\ar[r]^{\alpha}  & H^3_I(S_p)\ar[r]^{\beta}  & H^4_{I+pS}(S)\ar[r]^{\gamma} & H^4_I(S) \ar[r] & 0. 
}
\]   
Supposing that the last column is exact, we now show via a diagram chase that the third column is also exact:  Take $v\in H^4_{I+pS}(S)$.  
If $w=\gamma(v)$, then there exists $w'\in H^4_I(S)$ such that $pw'=w$ by our assumption. 
In addition, there exists $v'\in H^4_{I+pS}(S)$ such that $\gamma(v')=w'$. Since the diagram is commutative, we have that 
$\gamma(pv')=w = \gamma(v)$, so that  $v-pv'\in \Ker(\gamma)=\IM(\beta)$.  Thus, there exists $u\in H^3_I(S_p)$ such that
$\beta(u)=v-pv'$. Since multiplication by $p$ is an automorphism on
$H^3_I(S_p)$, there exist $u'$ such that $pu'=u$. 
Let $z=\beta(u')+v'$. Then
\[
pz=p\beta(u')+pv'=\beta(pu')+pv'=\beta(u)+pv' = v-pv'+pv'= v.\]  As $v$ was arbitrary, this shows that multiplication by $p$ on $H^4_{I+pS}(S)$ is surjective.  By our earlier discussion, this is a contradiction.
\end{proof}

With these results in hand, we are finally able to show that the answer to Lyubeznik's question (Question \ref{Q InjDim} in the introduction) is not always positive.  

\begin{theorem}[cf.\,Theorem \ref{Thm Counter-Example}]
\label{exactCounterexample: T}
Let $T$ denote the completion of $S$ at $\m$.  Then the 
 injective dimension of $H^4_{I T}(T) =  H^0_{\m T} H^4_{I T}(T)$ is one, while its support is zero-dimensional.
In particular, this local cohomology module provides a negative answer to both parts of Lyubeznik's question.  
\end{theorem}

\begin{proof}
As  $H^4_I(S)$ is supported only at $\m$ by Lemma \ref{Supp Zero}, it is naturally a module over the completion of $S$ at $\m$, $T$;
moreover, $H^4_{I T}(T) \cong H^4_I(S)$ as $T$-modules, and as a $T$-module, it is supported only at the maximal ideal $\m T$ of $T$.
As multiplication by $p$ is not surjective on $H^4_{I T}(T) = H^0_{\m T} H^4_{I T}(T) $  by Lemma \ref{multpnotsurj}, we know that this module is not injective by 
 Remark \ref{Rem Sub Injective}.
We can then apply \cite[Theorem 5.1]{Zhou} to see that
\[
1 \leq \InjDim_T H^4_{I T}(T)\leq \dim\Supp_T H^4_{I T}(T) + 1 = 0 + 1 = 1.\]


\vspace{-.55cm}

\end{proof}

\section*{Acknowledgments}
The authors began this work in 2013 during the program in Commutative Algebra at the Mathematical Sciences Research Institute (MSRI).
We thank MSRI for their hospitality and support, which includes Postdoctoral Fellowships for the first and fourth authors.
The authors are also grateful to the National Science Foundation (NSF) and the National Council of Science and Technology of Mexico (CONACyT)
for support: 
The first author is supported by a NSF Postdoctoral Research Fellowship (DMS-1304250), 
the second was partially supported by CONACyT grant 207063 and NSF grant DMS-1502282,
the third  by NSF grant DMS-1068190,
and the last by NSF grant DMS-1501404.


\bibliographystyle{alpha}
\bibliography{LyubeznikNumberInjectiveDimensionMixedChar}


\vspace{.6cm}

{\small
\noindent  \textsc{Department of Mathematics, University of Michigan, Ann Arbor, MI  48109} \\ \indent \emph{Email address}: {\tt dhernan@umich.edu}

\vspace{.25cm}

\noindent \small \textsc{Department of Mathematics, University of Virginia, Charlottesville, VA  22903} \\ \indent \emph{Email address}:  {\tt lcn8m@virginia.edu} 

\vspace{.25cm}

\noindent \small \textsc{Department of Mathematics \& Statistics, Georgia State University, Atlanta, GA 30303} \\ \indent  \emph{Email address}:  {\tt jperezvallejo@gsu.edu} 

\vspace{.25cm}

\noindent \small \textsc{Department of Mathematics, Kansas University, Lawrence, KS 66045} \\ \indent \emph{Email address}:  {\tt witt@ku.edu} 
}

\end{document}